\documentclass[12pt,reqno]{amsart}
\usepackage{amssymb}
\usepackage{amsmath, mathtools}

\usepackage{amsthm}

\usepackage{amscd}

\newcommand{\RNum}[1]{\uppercase\expandafter{\romannumeral #1\relax}}

\usepackage[T2A]{fontenc}
\usepackage[utf8]{inputenc}
\usepackage[english]{babel}

\input{int.def} 

\usepackage[sort]{cite}
\usepackage{tikz-cd}
\usetikzlibrary{cd}
\usepackage{dirtytalk}
\usepackage[linktoc=page, colorlinks, linkcolor=blue, citecolor=blue]{hyperref}

\usepackage{xcolor}
\usepackage{centernot}

\usepackage{enumitem}

\usepackage{pgfplots}
\usepackage{multicol}

\pgfplotsset{compat=1.17}

\makeatletter
\renewenvironment{proof}[1][\proofname]{%
  \par\vspace{\topsep}%
  \normalfont\topsep6\p@\@plus6\p@\relax
  \trivlist
  \item[\hskip\labelsep\itshape #1\@addpunct{.}]\ignorespaces
}{%
  \endtrivlist
}
\makeatother

\numberwithin{equation}{section}

\DeclarePairedDelimiterX \ip[2]{\langle}{\rangle}{#1,#2}
\DeclarePairedDelimiterXPP \Prob[1]{\mathbb{P}}\{\}{}{ #1} 
\DeclarePairedDelimiterXPP \Probevent[1]{\mathbb{P}}(){}{#1} 

\DeclareMathOperator{\E}{\mathbb{E}}

\newcommand{\one}{\mathbf{1}}

\def \R {\mathbb{R}}
\def \e {\varepsilon}

\usepackage[mathcal]{euscript}

\usepackage{titlesec}
\titleformat{\section}[runin]{\bfseries}{\thesection.}{3pt}{}[.]

\usepackage{geometry}
\newgeometry{vmargin={25mm}, hmargin={22mm,22mm}, footskip=10mm}   

\begin{document}

\title[Optimal discrepancy for triangles]{Optimal discrepancy for triangles}

\author{Gleb Smirnov}
\address{
Mathematical Sciences Institute, 
Australian National University, Canberra, Australia}
\email{gleb.smirnov@anu.edu.au}


\begin{abstract}
We prove that every $n$-point set in the plane has a two-coloring with discrepancy $O(n^{1/4})$ for all triangles. This improves the previous bound $O(n^{1/4}\sqrt{\log n})$ and matches the classical lower bound.
\end{abstract}

\maketitle

\setcounter{section}{0}


\section{Main results}\label{intro}
Let \(X=\{p_1,\ldots,p_n\}\subset\R^2\) be a set of points. We color each point \(+1\) or \(-1\). We want every triangle to contain nearly as many \(+1\)'s as \(-1\)'s. Encode a coloring by a vector \(\chi\in\{-1,1\}^n\): the point \(p_i\) gets color \(\chi_i\). For \(F\subset X\), write:
\[
\chi(F):=\sum_{p_i\in F}\chi_i.
\]
The \emph{combinatorial discrepancy} of \(\chi\) for triangles is \(\sup_T|\chi(X\cap T)|\), over all triangles \(T\). Our main result is the following.
\begin{theorem}\label{thm:triangles}
Every set \(X=\{p_1,\ldots,p_n\}\subset\R^2\) has a coloring \(\chi\in\{-1,1\}^n\) such that:
\[
|\chi(X\cap T)|\le C\,n^{1/4}\qquad\text{for all triangles }T,
\]
where \(C>0\) is an absolute constant.
\end{theorem}
For triangles, the best bound we are aware of is \(O(n^{1/4}\sqrt{\log n})\). It follows from more general results of Matou\v{s}ek et al.\ \cite{MWW}. Chazelle \cite[\S\,1.4]{Chaz} gives a simple exposition.
\smallskip%

Theorem~\ref{thm:triangles} is sharp. For a finite set, every halfplane cuts out the same points as some large triangle. So it suffices to prove the lower bound for halfplanes. For every \(n\), there is an \(n\)-point set on which every coloring has discrepancy \(\Omega(n^{1/4})\) for halfplanes. This is due to Alexander \cite{Alexander}. Chazelle, Matou\v{s}ek, and Sharir \cite{CMS} give a simpler proof.
\smallskip%

We deduce Theorem~\ref{thm:triangles} from a statement about halfplanes. A \emph{halfplane trace} is a set \(X\cap P\), where \(P\subset\R^2\) is a halfplane. Let \(\mathcal H(X)\) be the family of halfplane traces. Now let the points of \(X\) be ordered. For \(0\le k\le n\), set:
\[
X_k:=\{p_1,\ldots,p_k\}.
\]
\begin{theorem}\label{thm:prefix1}
Every ordered \(X=\{p_1,\ldots,p_n\}\subset\R^2\) has a coloring \(\chi\in\{-1,1\}^n\) such that:
\[
|\chi(F\cap X_k)|\le C\,n^{1/4}\qquad\forall F\in\mathcal H(X),\ \ 0\le k\le n,
\]
where \(C>0\) is an absolute constant.
\end{theorem}
For \(k=n\), this is Matou\v{s}ek's bound for halfplanes \cite{Mat-halfspaces}. The novelty is that one coloring controls all prefixes at once. We explain the step from halfplanes to triangles in \S\,\ref{sec:triangles}.
\smallskip%

We prove Theorem~\ref{thm:prefix1} by the Fisher information technique of \cite{Versh-Sm-3, Versh-Sm-4}. Consider the body of all \(z\in\R^n\) with \(|\sum_{p_i\in F}z_i|\le D\) for every \(F\in\mathcal H(X)\), where \(D\) is of order \(n^{1/4}\). This body is an intersection of slabs. The crucial step is to bound the Fisher information of densities supported on such intersections. We derive these bounds in \S\,\ref{sec:slabs}. They may be of independent interest. The proof is completed in \S\,\ref{chaining} and \S\,\ref{sec:prefix}.
\smallskip%

We now discuss related work.
\smallskip%

Theorem~\ref{thm:prefix1} is a vector balancing statement. For each \(p_i\), let \(u_i\in\{0,1\}^{\mathcal H(X)}\) indicate the halfplane traces that contain \(p_i\). Then Theorem~\ref{thm:prefix1} says that all prefix sums \(\sum_{i\le k}\chi_iu_i\) are small in \(\ell_\infty\).
\smallskip%

Vector balancing has seen rapid progress. For the Koml\'os problem, Bansal and Jiang \cite{BansalJiang} improved Banaszczyk's \(O(\log^{1/2}n)\) bound to \(O(\log^{1/4}n)\). Guo, Fang, and Lu \cite{GFL} obtained a constant bound. This settles the Koml\'os conjecture. For alternative proofs, see Karingula and Lovett \cite{KarLov} and Akbas and Sra \cite{AkbasSra-1, AkbasSra-2}. In the online version of the problem, the vectors are fixed in advance and revealed one by one. Each must be signed on arrival, so that all prefix sums stay small. Kulkarni, Reis, and Rothvoss \cite{Kulk-Reis-Roth} and Aden-Ali \cite{AdenAli} obtained optimal bounds. These improve earlier bounds of Alweiss, Liu, and Sawhney \cite{AL-Saw} and Bansal et al.\ \cite{Bans-JMSS}. Finally, Altschuler and Tikhomirov \cite{AltTikh-1} solved the online Spencer problem, where the vectors are only bounded in \(\ell_\infty\).
\smallskip%

Our proof uses the partial coloring method. It goes back to Beck \cite{Beck-3} and Spencer \cite{Spencer-6}. Constructive versions are due to Bansal \cite{Bansal-10} and to Lovett and Meka \cite{LovMeka}. These methods need the convex body to be large, in entropy or in Gaussian measure. Our mechanism is different. We need instead a density supported in the body with small Fisher information.
\smallskip%

For combinatorial and Lebesgue discrepancy of geometric set systems, see the monographs \cite{Mat, Chaz, Beck-Chen}. Like Matou\v{s}ek's proof for halfplanes \cite{Mat-halfspaces}, ours uses chaining and the packing lemma for halfplanes. Mustafa~\cite{Mustafa} gives a simple proof of that lemma. For axis-parallel boxes, see Matou\v{s}ek et al.\ \cite{MNT} and Nikolov \cite{Nikolov}. For the Lebesgue discrepancy of convex polygons, see Chen and Travaglini~\cite{ChenTrav}. For rectangles with restricted rotations, see Bilyk and Mastrianni~\cite{BilykMast}.
\smallskip%

Beck's transference principle \cite{Beck-1} turns upper bounds on combinatorial discrepancy into upper bounds on Lebesgue discrepancy. By Theorem~\ref{thm:triangles}, for every \(n\), some \(n\)-point set in \([0,1]^2\) has Lebesgue discrepancy \(O(n^{1/4})\) for rotated rectangles. This improves Beck's bound \(O(n^{1/4}\sqrt{\log n})\) \cite{Beck-Chen} and matches his 
lower bound \cite{Beck-2}.
\smallskip%

\textbf{Acknowledgment and AI assistance.} 
The author thanks OpenAI's ChatGPT for useful insights. ChatGPT designed a proof of 
Proposition \ref{slabs}, although this is not the proof used in the present paper. ChatGPT carried out the estimates in the proof of 
Proposition \ref{prop:halfplanes}. Finally, ChatGPT helped edit the manuscript.

\section{Fisher cost of intersecting slabs}\label{sec:slabs}
 
Let \(K\subset\R^n\) be a convex body. We seek a function \(\psi\in C_c^\infty(\R^n)\) with
\[
\int_{\R^n}\psi^2=1,\qquad
\operatorname{supp}\psi\subset K,
\]
and with small energy
\[
\mathcal E(\psi):=\int_{\R^n}\|\nabla\psi\|_2^2.
\]
The energy controls the number of discarded terms in Theorem~\ref{thm:SV}. In \cite{Versh-Sm-3}, we denoted it by \(\operatorname{tr}M_\psi\). The bodies \(K\) we need are intersections of slabs. Here, we bound \(\mathcal E(\psi)\) for such intersections.
\begin{proposition}\label{slabs}
Let \(v_1,\ldots,v_M\in\R^n\) be unit vectors, and let \(r_1,\ldots,r_M>0\). Assume that the body
\[
K:=\bigl\{z\in\R^n:\ |\langle v_j,z\rangle|\le r_j\quad \forall j\bigr\}
\]
is bounded. Then there is \(\psi\in C_c^\infty(\R^n)\) such that:
\[
\int_{\R^n} \psi^2 = 1,\qquad
\operatorname{supp}\psi\subset K,
\]
and such that:
\begin{equation}\label{eq:slab-fisher}
\mathcal E(\psi)\le
\frac{n}{4} + C 
\sum_{j=1}^M\frac{1}{r_j^2}\, 
e^{-r_j^2/32},
\end{equation}
where \(C>0\) is an absolute constant.
\end{proposition}
The rest of this section proves Proposition~\ref{slabs}. The proof uses quadratic partitions of unity and the IMS localization formula. See Simon \cite[\S\,3]{Simon-IMS} for background and applications. To keep the exposition self-contained, here we give proofs of 
everything we use.
\smallskip%
 
\emph{Quadratic partitions.} Cover \(\R\) by the intervals:
\[
I_k:=\Bigl[\frac{3k}{4}-\frac12,\ \frac{3k}{4}+\frac12\Bigr],\qquad k\in\mathbb Z.
\]
Each \(I_k\) meets \(I_{k-1}\) and \(I_{k+1}\), and no other interval. The middle part \(\bigl(\frac{3k}{4}-\frac14,\ \frac{3k}{4}+\frac14\bigr)\) of \(I_k\) meets no other interval. Take a quadratic partition of unity subordinate to this cover. That is, take smooth functions \(\eta_k\) such that:
\[
\operatorname{supp}\eta_k\subset I_k,\qquad \sum_{k\in\mathbb Z}\eta_k^2=1.
\]
We take it translation invariant:
\[
\eta_k(t)=\eta_0(t-3k/4).
\]
To build such a partition, fix \(\beta_0\in C^\infty(\R)\) with \(\beta_0>0\) on \((-1/2,1/2)\) and \(\beta_0=0\) outside. Then set:
\[
\beta_k(t):=\beta_0(t-3k/4),\qquad
\eta_k(t):=\frac{\beta_k(t)}{\sqrt{\sum_\ell\beta_\ell(t)^2}}.
\]
More generally, a quadratic partition of 
unity in \(\R^n\) is a locally finite family \(A_a\) of smooth functions with 
\[
\sum_aA_a^2=1.
\]
The product of two such partitions is again one. Moreover, their gradient energies add up.
\begin{lemma}\label{lem:product}
Let \(A_a\) and \(B_b\) be quadratic partitions of unity on \(\R^n\). Then \(A_a B_b\) is also 
a quadratic partition of unity, and:
\[
\sum_{a,b}\|\nabla(A_aB_b)\|_2^2=\sum_a\|\nabla A_a\|_2^2+\sum_b\|\nabla B_b\|_2^2.
\]
\end{lemma}
\begin{proof}
Using
\begin{multline*}
\|\nabla(A_aB_b)\|_2^2 = B_b^2\|\nabla A_a\|_2^2 + A_a^2\|\nabla B_b\|_2^2 + 2A_aB_b\langle\nabla A_a,\nabla B_b
\rangle \\  
= B_b^2\|\nabla A_a\|_2^2 + A_a^2\|\nabla B_b\|_2^2 + \frac12 \langle\nabla A_a^2,\nabla B_b^2
\rangle,
\end{multline*}
and 
\[
\sum_a A_a^2 = \sum_b B_b^2 = 1,
\]
we find:
\[
\sum_{a, b} \|\nabla(A_aB_b)\|_2^2 = 
\sum_{a} \|\nabla A_a\|_2^2 + \sum_{b} \|\nabla B_b\|_2^2 + \frac12 \langle\nabla 1,\nabla 1
\rangle,
\]
and the proof follows. \qed
\end{proof}
By induction, Lemma~\ref{lem:product} extends to products of finitely many quadratic partitions.
\smallskip%
 
\emph{Slab partitions.} For each \(j = 1, \ldots, M\), set:
\[
t_j(z):=\frac{\langle v_j,z\rangle}{r_j},\qquad Q_{j,k}(z):=\eta_k\bigl(t_j(z)\bigr),\quad k\in\mathbb Z.
\]
For each fixed \(j\), the family \(Q_{j,k}\) 
is a quadratic partition of unity on \(\R^n\). The function \(Q_{j,k}\) is supported in the slab \(\{t_j\in I_k\}\). Since \(\nabla t_j = v_j/r_j\) and \(\|v_j\|_2=1\), it follows that:
\[
\sum_k\|\nabla Q_{j,k}\|_2^2 = 
\frac{1}{r_j^2}\,\sum_k |\dot{\eta}_k(t_j)|^2.
\]
To shorten notation, set:
\[
D(t) := \sum_k |\dot{\eta}_k(t)|^2.
\]
Since \(D\) is continuous and \(3/4\)-periodic, 
\[
\sup D = C_0 < \infty.
\]
For \(\alpha = (\alpha_1, \ldots, \alpha_M)\in\mathbb Z^M\), set:
\[
P_\alpha:=\prod_{j=1}^MQ_{j,\alpha_j},\qquad
L_\alpha:=\bigl\{z\in\R^n:\ \bigl|\langle v_j,z\rangle-\tfrac34\alpha_jr_j\bigr|\le 
\tfrac12 r_j\ \ \forall j\bigr\}.
\]
By Lemma~\ref{lem:product}, \(P_\alpha\) is 
a quadratic partition of unity such that:
\begin{equation}\label{total_energy}
\sum_\alpha\|\nabla P_\alpha\|_2^2 = \sum_{j=1}^M\frac{D(t_j)}{r_j^2},\qquad
\operatorname{supp}P_\alpha\subset L_\alpha.
\end{equation}
The following is (a version of) the IMS localization formula.
\begin{lemma}
Let \(\phi\in C^\infty(\R^n)\) with 
\[
\int_{\R^n} \phi^2 = 1,\qquad 
\mathcal E(\phi)<\infty.
\]
Let \(Z\) be a random vector with density \(\phi^2\). Then there is \(\psi\in C_c^\infty(\R^n)\) such that:
\[
\int_{\R^n} \psi^2 = 1,\qquad 
\operatorname{supp}\psi\subset K,
\]
and such that:
\begin{equation}\label{eq:localize}
\mathcal E(\psi)\le\mathcal E(\phi)+C_0\sum_{j=1}^M\frac{1}{r_j^2}\,\mathbb P\Bigl(|\langle v_j,Z\rangle|\ge\frac{r_j}4\Bigr).
\end{equation}
\end{lemma}
\begin{proof}
Set \(h_\alpha:=\phi P_\alpha\). Each \(h_\alpha\) is smooth and supported in \(L_\alpha\). Since 
\[
\sum_\alpha P_\alpha^2=1,\qquad 
\sum_\alpha P_\alpha\nabla P_\alpha=0,
\]
it follows that:
\[
\sum_{\alpha} \int_{\R^n} h_{\alpha}^2 = 
\int_{\R^n} \phi^2 = 1,
\]
and, pointwise,
\[
\sum_\alpha\|\nabla h_\alpha\|_2^2=\|\nabla\phi\|_2^2+\phi^2\sum_\alpha\|\nabla P_\alpha\|_2^2  
\stackrel{\eqref{total_energy}}{=} 
\|\nabla\phi\|_2^2+\phi^2 \sum_{j=1}^M\frac{D(t_j)}{r_j^2}.
\]
Consequently,
\[
\sum_\alpha\mathcal E(h_\alpha)=\mathcal E(\phi) + \sum_{j=1}^M\frac{1}{r_j^2}\,\E D\bigl(t_j(Z)\bigr).
\]
Setting
\[
B := \mathcal E(\phi) + \sum_{j=1}^M\frac{1}{r_j^2}\,\E D\bigl(t_j(Z)\bigr),\qquad 
f_{\alpha} = \frac{h_{\alpha}}{\sqrt{\int_{\R^n} h_{\alpha}^2}},
\]
we get:
\[
\sum_\alpha \mathcal E(f_\alpha) \int_{\R^n} h_{\alpha}^2 = B.
\]
Since 
\[
\sum_{\alpha} \int_{\R^n} h_{\alpha}^2 = 1,
\]
there exists \(\alpha\) such that:
\[
\mathcal E(f_\alpha) \le B.
\]
To bound \(B\), first note:
\[
D(t)\le C_0\,\mathbf{1}_{\{|t|\ge 1/4\}}.
\]
Indeed, \(D\le C_0\) by the definition of \(C_0\). On \((-\tfrac14,\tfrac14)\), only \(\eta_0\) is nonzero. Hence, \(\eta_0=1\) there, and \(D\) vanishes. Since \(t_j(Z)=\langle v_j,Z\rangle/r_j\), it follows that:
\[
\E D\bigl(t_j(Z)\bigr)\le C_0\,\mathbb P\Bigl(|\langle v_j,Z\rangle|\ge\frac{r_j}4\Bigr).
\]
For \(z,z'\in L_\alpha\), we have \(|\langle v_j,z-z'\rangle|\le r_j\) for all \(j\); hence, \(L_\alpha-L_\alpha\subset K\). Fix \(z_0\in L_\alpha\), and set:
\[
\psi(z):= f_{\alpha}(z + z_0).
\]
Then:
\[
\int\psi^2=1,\quad 
\mathcal E(\psi)\le B,\quad 
\operatorname{supp}\psi\subset 
L_\alpha-z_0\subset K,
\]
and the proof follows. \qed
\end{proof}
 
\begin{proof}[Proof of Proposition~\ref{slabs}]
Let \(Z\sim N(0,I_n)\), and let 
\[
\phi(z) := (2\pi)^{-n/4}e^{-\|z\|_2^2/4}
\]
be the square root 
of its density. Then 
\(\|\nabla\phi\|_2^2=\tfrac14\|z\|_2^2\,\phi^2\); hence:
\[
\mathcal E(\phi)=\frac14\,\E\|Z\|_2^2=\frac n4.
\]
Since \(\|v_j\|_2=1\), \(\langle v_j,Z\rangle\) is a standard Gaussian; hence:
\[
\mathbb P\Bigl(|\langle v_j,Z\rangle|\ge\frac{r_j}4\Bigr)\le2\,e^{-r_j^2/32}.
\]
By \eqref{eq:localize},
\[
\mathcal E(\psi) \le \frac{n}{4} + 2C_0\sum_{j=1}^M\frac{1}{r_j^2}\,e^{-r_j^2/32},
\]
and the proof follows. \qed
\end{proof}

\section{Chaining}\label{chaining}
Let \(X=\{p_1,\ldots,p_n\}\subset\R^2\). 
A \emph{halfplane trace} is a set \(X\cap P\), where \(P\subset\R^2\) is a halfplane. Let \(\mathcal H(X)\) denote the family of halfplane traces. For \(F\subset X\), let \(\one_F\in\R^n\) be its indicator vector: 
\[
(\one_F)_i = \begin{cases}
1, & \text{if \(p_i\in F\)},\\
0, & \text{if \(p_i\notin F\)}.
\end{cases}
\]
Thus:
\[
\langle\one_F,z\rangle=\sum_{p_i\in F}z_i\quad \text{for \(z\in\R^n\)}.
\]
For \(D>0\), set:
\[
K_D(X):=\bigl\{z\in\R^n:\ |\langle\one_F,z\rangle|\le D\ \text{for all }F\in\mathcal H(X)\bigr\}.
\]
\(K_D(X)\) is convex and centrally symmetric. Let us show that \(K_D(X)\) is bounded. Choose a direction \(u\in\R^2\) such that the numbers \(\langle u, p_i\rangle\) are distinct. Order the points of \(X\) by \(\langle u,\cdot\rangle\). For each \(k\), the first \(k\) points in this order form a halfplane trace \(G_k\). Every singleton \(\{p_i\}\) is then a difference \(G_k\setminus G_{k-1}\). Therefore, \(|z_i|\le2D\) for all \(z\in K_D(X)\); that is,
\[
K_D(X)\subset[-2D,2D]^n.
\]
\smallskip%

We use the following known packing result. See \cite[\S\,5.3]{Mat}, \cite{ChazWelzl}, and \cite{Mustafa}.
\begin{theorem}\label{thm:packing}
Let \(\delta > 0\). A family \(\mathcal F\subset\mathcal H(X)\) is \emph{\(\delta\)-separated} if \(|F\triangle F'|>\delta\) for all distinct \(F,F'\in\mathcal F\). There is an absolute constant \(C_1\) such that for every \(0<\delta\le n\) and every \(\delta\)-separated \(\mathcal F\subset\mathcal H(X)\):
\[
|\mathcal F|\le C_1\,(n/\delta)^2.
\]
\end{theorem}
For \(\delta<1\), Theorem~\ref{thm:packing} is just the bound \(|\mathcal H(X)|=O(n^2)\).
\smallskip%

\emph{Nets.} We use Theorem~\ref{thm:packing} to build a hierarchy of nets in \(\mathcal H(X)\). This greedy construction is well known. For \(F,G\in\mathcal H(X)\), set:
\[
d(F,G):=|F\triangle G|.
\]
Then:
\[
\|\one_F-\one_G\|_2^2=d(F,G).
\]
Set:
\[
L:=\lceil\log_2n\rceil+1,\qquad
\delta_i:=n2^{-i}\qquad 0\le i\le L.
\]
Set \(\mathcal F_0:=\{\varnothing\}\). For \(1\le i\le L\), let \(\mathcal F_i\subset\mathcal H(X)\) be a maximal \(\delta_i\)-separated family. By Theorem~\ref{thm:packing}, \(|\mathcal F_i|\le C_1\,4^i\).
\smallskip%

By maximality, each \(\mathcal F_i\) is a \(\delta_i\)-net: every \(G\in\mathcal H(X)\) satisfies \(d(G,F)\le\delta_i\) for some \(F\in\mathcal F_i\). This also holds for \(\mathcal F_0=\{\varnothing\}\), since \(d(G,\varnothing)=|G|\le n=\delta_0\). Hence, each \(F\in\mathcal F_i\), \(i\ge1\), has a \emph{parent} \(p(F)\in\mathcal F_{i-1}\) with:
\[
d(F,p(F))\le \delta_{i-1}=n2^{1-i}.
\]
Finally, \(\delta_L\le1/2\). Since \(d\) is integer-valued, \(\mathcal F_L=\mathcal H(X)\).
\smallskip%

\emph{Increments.} For \(1\le i\le L\), set:
\[
\mathcal V_i:=\bigl\{\one_F-\one_{p(F)}:\ F\in\mathcal F_i\bigr\}\setminus\{0\}.
\]
Fix \(H\in\mathcal H(X)\). Since \(\mathcal F_L=\mathcal H(X)\), we may set \(F_L:=H\). Then set:
\[
F_{i-1}:=p(F_i)\quad \text{for \(i = L, \ldots, 1\).}
\]
The increments telescope:
\begin{equation}\label{eq:telescope}
\one_H=\sum_{i=1}^L\bigl(\one_{F_i}-\one_{F_{i-1}}\bigr).
\end{equation}
Each \(\one_{F_i}-\one_{F_{i-1}}\) is either \(0\) or in \(\mathcal V_i\). Along with \eqref{eq:telescope}, we use the following properties of the increments:
\begin{equation}\label{eq:increments}
|\mathcal V_i|\le C_1\,4^i,\qquad
\|v\|_2^2\le \delta_{i-1}\quad \forall v\in\mathcal V_i.
\end{equation}
\begin{proposition}\label{prop:halfplanes}
For each \(D > 0\), there exists \(\psi_D\in C_c^\infty(\R^n)\) such that:
\[
\int_{\R^n}\psi_D^2=1,\qquad
\operatorname{supp}\psi_D\subset K_{D}(X),
\]
and such that:
\begin{equation}\label{eq:halfplane-energy}
\mathcal E(\psi_D)\le C_2\,\frac{n^{3/2}}{D^2},
\end{equation}
where \(C_2 > 0\) is an absolute constant.
\end{proposition}
\begin{proof}
We do not apply Proposition~\ref{slabs} to \(K_D(X)\) directly. Instead, we introduce:
\[
K':=\bigl\{z\in\R^n:\ |\langle v,z\rangle|\le b_i\ \text{for all }v\in\mathcal V_i\text{ and all }i\bigr\},
\]
where the widths \(b_i>0\) are chosen below. By \eqref{eq:telescope}, for \(z\in K'\) and \(H\in\mathcal H(X)\):
\[
|\langle\one_H,z\rangle| \le \sum_{i=1}^L|\langle\one_{F_i}-\one_{F_{i-1}},z\rangle| \le\sum_{i=1}^Lb_i.
\]
Hence:
\begin{equation}\label{eq:chain-support}
K'\subset K_{b}(X),\qquad b:=\sum_{i=1}^Lb_i.
\end{equation}
In particular, \(K'\) is bounded.
\smallskip%

\emph{Cost.} Apply Proposition~\ref{slabs} to \(K'\): the unit vectors are \(v/\|v\|_2\), \(v\in\mathcal V_i\), and the width of \(v/\|v\|_2\) is \(b_i/\|v\|_2\). Let \(\psi\) be the resulting function. By \eqref{eq:slab-fisher},
\[
\mathcal E(\psi)\le\frac n4+C\sum_{i=1}^L\sum_{v\in\mathcal V_i}c\Bigl(\frac{b_i^2}{\|v\|_2^2}\Bigr),
\qquad\text{where }\ c(w):=\frac1w\,e^{-w/32}.
\]
\(c(w)\) is decreasing. By \eqref{eq:increments}, \(\|v\|_2^2\le\delta_{i-1}\) for \(v\in\mathcal V_i\); hence, each term is at most \(c(b_i^2/\delta_{i-1})\). Also by \eqref{eq:increments}, the inner sum has at most \(C_1\,4^i\) terms. Therefore:
\begin{equation}\label{eq:level-sum}
\mathcal E(\psi)\le\frac n4+O\Bigl(\sum_{i=1}^L4^i\,c\Bigl(\frac{b_i^2}{\delta_{i-1}}\Bigr)\Bigr).
\end{equation}
We call the \(i\)-th term of this sum the (Fisher) cost of level \(i\). We want to arrange
\[
\sum_i b_i = O(n^{1/4}),
\]
while keeping the total cost \(O(n)\). Choose \(i_*\) such that:
\[
\sqrt n\le2^{i_*}<2\sqrt n.
\]
Set:
\[
b_i:=n^{1/4}\,2^{-|i-i_*|/4}.
\]
The widths peak at \(i=i_*\) and decay geometrically in \(|i-i_*|\). Hence:
\[
\sum_i b_i\le n^{1/4}\sum_{j\in\mathbb Z}2^{-|j|/4}=O(n^{1/4}).
\]
Let us bound the cost of level \(i\). Write:
\[
j:=i-i_*.
\]
Since \(2^{i_*}<2\sqrt n\),
\[
4^i=4^{j}4^{i_*}<4n\,4^{j}.
\]
Since \(2^{i_*}\ge\sqrt n\) and \(b_i^2=\sqrt n\,2^{-|j|/2}\),
\[
\delta_{i-1}=n2^{1-i}=2n\,2^{-i_*}2^{-j}\le2\sqrt n\,2^{-j},\qquad
\frac{b_i^2}{\delta_{i-1}}\ge2^{\,j-|j|/2-1}.
\]
Since \(c\) is decreasing, the cost of level \(i\) is at most \(4n\,4^{j}\,c(2^{\,j-|j|/2-1})\). Then:
\[
\sum_{i=1}^L4^i\,c\Bigl(\frac{b_i^2}{\delta_{i-1}}\Bigr)\le
4n\sum_{j\le0}4^{j}\,c\bigl(2^{3j/2-1}\bigr)+4n\sum_{j>0}4^{j}\,c\bigl(2^{j/2-1}\bigr).
\]
\emph{Few but costly: \(j\le0\).} These levels have few slabs, since \(4^j\le1\). The slabs are costly, since \(2^{3j/2-1}\le1/2\). We bound these terms as follows:
\[
\sum_{j\le0}4^{j}\,c\bigl(2^{3j/2-1}\bigr)=\sum_{j\le0}2^{2j}\,2^{1-3j/2}\,e^{-2^{3j/2}/64}\le2\sum_{j\le0}2^{j/2}=O(1).
\]
\emph{Many but cheap: \(j>0\).} These levels have many slabs, since \(4^j>1\). We bound these terms as follows:
\[
\sum_{j>0}4^{j}\,c\bigl(2^{j/2-1}\bigr)=\sum_{j>0}4^{j}\,2^{1-j/2}\,e^{-2^{j/2}/64}=2\sum_{j>0}2^{3j/2}\,e^{-2^{j/2}/64}=O(1).
\]
Therefore:
\[
\sum_{i=1}^L4^i\,c\Bigl(\frac{b_i^2}{\delta_{i-1}}\Bigr)=O(n).
\]
By \eqref{eq:level-sum},
\[
\mathcal E(\psi)\le\frac n4+O(n)=O(n).
\]
\emph{Scaling.} Let us rescale \(\psi\). Set:
\[
\psi_D(z):=(b/D)^{n/2}\,\psi(bz/D).
\]
Then:
\[
\int_{\R^n}\psi_D^2=1,\qquad 
\operatorname{supp}\psi_D \subset K_D(X).
\]
Since \(b=O(n^{1/4})\),
\[
\mathcal E(\psi_D)=\frac{b^2}{D^2}\,\mathcal E(\psi)=O\Bigl(\frac{b^2n}{D^2}\Bigr)=O\Bigl(\frac{n^{3/2}}{D^2}\Bigr),
\]
and the proof follows. \qed
\end{proof}

\section{Prefix halfplanes}\label{sec:prefix}
Recall the setting of \S\,\ref{chaining}. Let \(X=\{p_1,\ldots,p_n\}\subset\R^2\) be a set of points, and let \(\mathcal H(X)\) be the family of halfplane traces. For \(F\subset X\), let \(\one_F\in\R^n\) be its indicator vector. For \(D>0\), set again:
\[
K_D(X)=\bigl\{z\in\R^n:\ |\langle\one_F,z\rangle|\le D\ \text{for all }F\in\mathcal H(X)\bigr\}.
\]
Now let the points of \(X\) be ordered. For \(0\le k\le n\), set:
\[
X_k:=\{p_1,\ldots,p_k\}.
\]
For \(\chi\in\R^n\) and \(F\subset X\), write:
\[
\chi(F):=\langle\one_F,\chi\rangle=\sum_{p_i\in F}\chi_i.
\]
\begin{theorem}\label{thm:prefix}
Every ordered \(X=\{p_1,\ldots,p_n\}\) has a coloring \(\chi\in\{-1,1\}^n\) with:
\[
|\chi(F\cap X_k)|=O(n^{1/4})\qquad\forall F\in\mathcal H(X),\ \ 0\le k\le n.
\]
\end{theorem}
The proof relies on the following result of \cite{Versh-Sm-3}.
\begin{theorem}\label{thm:SV}
Let \(K\subset\R^n\) be convex and centrally symmetric. Let \(\psi\in C_c^\infty(\R^n)\) be such that:
\[
\int_{\R^n}\psi^2=1,\qquad 
\operatorname{supp}\psi\subset K.
\]
Let \(v_1,\ldots,v_m\in\R^n\) be any vectors, and let \(\e_1,\ldots,\e_m\) be iid fair signs. Then there is an online algorithm that discards some terms \(\e_iv_i\) so that for each \(k\),
\[
\sum_{i\le k\ \mathrm{accepted}}\e_iv_i\in2K
\quad\text{deterministically},
\]
and
\[
\E\,\#\{\text{\rm discarded}\}\le\sum_{i=1}^m\Bigl(\int_{\R^n}\langle\nabla\psi,v_i\rangle^2\Bigr)^{1/2}.
\]
\end{theorem}
First, we do a partial coloring.
\begin{lemma}\label{lem:partial}
Every ordered \(X=\{p_1,\ldots,p_n\}\) has a partial coloring \(\chi\in\{-1,0,1\}^n\) with at most \(n/2\) zeros such that:
\[
|\chi(F\cap X_k)|\le C_3\,n^{1/4}\qquad\forall F\in\mathcal H(X),\ \ 0\le k\le n,
\]
where \(C_3>0\) is an absolute constant.
\end{lemma}
\begin{proof}
Let \(D > 0\). Apply Theorem~\ref{thm:SV} to \(K=K_D(X)\). Let \(\psi=\psi_D\) be given by Proposition~\ref{prop:halfplanes}, and let \(v_i:=e_i\), \(i=1,\ldots,n\). Set:
\[
\chi_i:=\begin{cases}
\e_i, & \text{if \(\e_ie_i\) is accepted},\\
0, & \text{if \(\e_ie_i\) is discarded}.
\end{cases}
\]
Then:
\[
\forall k\qquad \sum_{i\le k}\chi_ie_i\in2K_D(X).
\]
That is,
\[
|\chi(F\cap X_k)|\le2D\qquad\forall F\in\mathcal H(X),\ \ 0\le k\le n.
\]
Let us count the zeros. By Theorem~\ref{thm:SV} and Cauchy-Schwarz,
\[
\E\,\#\{\chi_i=0\}
\le\sum_{i=1}^n\Bigl(\int_{\R^n}(\partial_i\psi_D)^2\Bigr)^{1/2}
\le\sqrt n\,\mathcal E(\psi_D)^{1/2}.
\]
By \eqref{eq:halfplane-energy},
\[
\E\,\#\{\chi_i=0\}\le\sqrt{C_2}\,\frac{n^{5/4}}{D}.
\]
Set \(D:=2\sqrt{C_2}\,n^{1/4}\). Then:
\[
\E\,\#\{\chi_i=0\}\le \frac{n}{2}.
\]
Hence, some outcome of the signs leaves at most \(n/2\) zeros. This completes the proof. \qed
\end{proof}

\begin{proof}[Proof of Theorem~\ref{thm:prefix}]
The iteration of partial colorings goes back to Spencer \cite{Spencer-6}. We iterate Lemma~\ref{lem:partial} as in \cite[\S\,2]{Versh-Sm-4}. First, apply it to \(X\). Then apply it to the points left uncolored, in the induced order. Continue until all points are colored. Each round colors at least half of the remaining points. Hence, round \(r\ge0\) is applied to at most \(n2^{-r}\) points. Let \(\chi\in\{-1,1\}^n\) be the final coloring.
\smallskip%

Fix \(F\in\mathcal H(X)\) and 
\(0\le k\le n\). Let \(Y_r \subset X\) be the set of points that receive their \((\pm 1)\)-color at round \(r\). The sets \(Y_r\) partition \(X\). Round \(r\) is applied to \(Y_r\cup Y_{r+1}\cup\cdots\), with the order induced from \(X\). By Lemma~\ref{lem:partial},
\[
|\chi(F\cap Y_r\cap X_k)|\le C_3\,(n2^{-r})^{1/4}.
\]
Consequently,
\[
|\chi(F \cap X_k)| \le 
\sum_{r\ge0} |\chi(F \cap Y_r \cap X_k)| \le 
C_3\,n^{1/4}\sum_{r\ge0}2^{-r/4} = O(n^{1/4}),
\]
and the proof follows. \qed
\end{proof}

\section{Triangles}\label{sec:triangles}
Let \(X=\{p_1,\ldots,p_n\}\subset\R^2\) be a set of points. For \(\chi\in\R^n\) and \(F\subset X\), write again:
\[
\chi(F)=\sum_{p_i\in F}\chi_i.
\]
\begin{theorem}
Every \(X=\{p_1,\ldots,p_n\}\) has a coloring \(\chi\in\{-1,1\}^n\) with:
\[
|\chi(X\cap T)|=O(n^{1/4})\qquad\text{for all triangles }T.
\]
\end{theorem}
\begin{proof}
The proof reduces triangles to corners.  For a nonvertical line \(\ell\colon y=\ell(x)\) and \(c\in\R\), set:
\[
\Gamma(\ell,c):=\bigl\{(x,y)\in\R^2:\ x\le c,\ \ y\le\ell(x)\bigr\}.
\]
Order \(X\) by the first coordinate. Let \(\chi\) be the coloring from Theorem~\ref{thm:prefix} for this order.
\colorlet{rgn}{blue!22}
\colorlet{rgnedge}{blue!55!black}
\tikzset{
  pic/.style={scale=0.55,baseline={([yshift=-.5ex]current bounding box.center)}},
  fade/.style={top color=rgn,bottom color=rgn!0,draw=none},
  solid region/.style={fill=rgn,draw=rgnedge,line width=0.6pt,line join=round},
  edge/.style={draw=rgnedge,line width=0.6pt,line join=round},
  ghost/.style={black!40,thin},
  guide/.style={black!45,densely dashed,thin},
  lab/.style={font=\scriptsize,inner sep=1.5pt}}
\newcommand{\picbox}{\useasboundingbox (-0.1,-0.55) rectangle (4.45,3);}
\newcommand{\guide}[2]{\draw[guide] (#1,0)--(#1,3);\node[lab,anchor=north] at (#1,0) {\(#2\)};}
\newcommand{\ghostline}{\draw[ghost] (0,0.9)--(4,2.5);\node[lab,anchor=west] at (4,2.5) {\(\ell\)};}
\newcommand{\op}[1]{\;#1\;}
\medskip%

\emph{Corners.} \(X\cap\Gamma(\ell,c)\) is a halfplane trace intersected with a prefix:
\[
\begin{tikzpicture}[pic]\picbox
\shade[fade] (0,0)--(2.6,0)--(2.6,1.94)--(0,0.9)--cycle;
\ghostline\guide{2.6}{c}
\draw[edge] (0,0.9)--(2.6,1.94)--(2.6,0);
\end{tikzpicture}
\op{=}
\begin{tikzpicture}[pic]\picbox
\shade[fade] (0,0)--(4,0)--(4,2.5)--(0,0.9)--cycle;
\ghostline
\draw[edge] (0,0.9)--(4,2.5);
\end{tikzpicture}
\op{\cap}
\begin{tikzpicture}[pic]\picbox
\shade[fade] (0,0)--(2.6,0)--(2.6,3)--(0,3)--cycle;
\guide{2.6}{c}
\draw[edge] (2.6,3)--(2.6,0);
\end{tikzpicture}
\]
By Theorem~\ref{thm:prefix},
\begin{equation}\label{eq:corner}
|\chi(X\cap \Gamma(\ell,c))|=O(n^{1/4}).
\end{equation}
\smallskip%

\emph{Interfixes.} Consider the domain below 
\(\ell\) and between the vertical lines \(x=a\) and \(x=b\); it is a difference of two corners:
\[
\begin{tikzpicture}[pic]\picbox
\shade[fade] (1,0)--(3,0)--(3,2.1)--(1,1.3)--cycle;
\ghostline\guide{1}{a}\guide{3}{b}
\draw[edge] (1,0)--(1,1.3)--(3,2.1)--(3,0);
\end{tikzpicture}
\op{=}
\begin{tikzpicture}[pic]\picbox
\shade[fade] (0,0)--(3,0)--(3,2.1)--(0,0.9)--cycle;
\ghostline\guide{3}{b}
\draw[edge] (0,0.9)--(3,2.1)--(3,0);
\end{tikzpicture}
\op{-}
\begin{tikzpicture}[pic]\picbox
\shade[fade] (0,0)--(1,0)--(1,1.3)--(0,0.9)--cycle;
\ghostline\guide{1}{a}
\draw[edge] (0,0.9)--(1,1.3)--(1,0);
\end{tikzpicture}
\]
Its trace is a halfplane trace intersected with an interfix \(X_k-X_j\). By \eqref{eq:corner}, its \(\chi\) is \(O(n^{1/4})\).
\smallskip%

\emph{Triangles with a vertical side.} Such a triangle is a difference of two domains of the previous kind:
\[
\begin{tikzpicture}[pic]\picbox
\draw[ghost] (0.3,2.985)--(4,0.95); \draw[ghost] (0.3,0.42)--(4,1.9);
\node[lab,anchor=east] at (0.3,2.985) {\(\ell_2\)}; \node[lab,anchor=east] at (0.3,0.42) {\(\ell_1\)};
\guide{1}{a}\guide{3}{b}
\draw[solid region] (1,0.7)--(1,2.6)--(3,1.5)--cycle;
\end{tikzpicture}
\op{=}
\begin{tikzpicture}[pic]\picbox
\shade[fade] (1,0)--(3,0)--(3,1.5)--(1,2.6)--cycle;
\draw[ghost] (0.3,2.985)--(4,0.95);
\guide{1}{a}\guide{3}{b}
\draw[edge] (1,0)--(1,2.6)--(3,1.5)--(3,0);
\node[lab,anchor=east] at (0.3,2.985) {\(\ell_2\)};
\end{tikzpicture}
\op{-}
\begin{tikzpicture}[pic]\picbox
\shade[fade] (1,0)--(3,0)--(3,1.5)--(1,0.7)--cycle;
\draw[ghost] (0.3,0.42)--(4,1.9);
\guide{1}{a}\guide{3}{b}
\draw[edge] (1,0)--(1,0.7)--(3,1.5)--(3,0);
\node[lab,anchor=east] at (0.3,0.42) {\(\ell_1\)};
\end{tikzpicture}
\]
For such a triangle, its \(\chi\) is again \(O(n^{1/4})\). 
\smallskip%

\emph{Triangles.} Every triangle \(T\) is a union of 
two triangles with vertical sides. Perturbing \(T\) slightly, we may assume that no point of \(X\) lies on the boundaries in the pictures. Hence, 
\(\chi(X\cap T)\) is the sum of two terms, each \(O(n^{1/4})\). \qed
\end{proof}

\bibliographystyle{plain}
\bibliography{ref}

\end{document}